\documentclass{amsproc}
\usepackage{a4,amsmath,amssymb,amsfonts,amsthm}
\usepackage{amsrefs}
\usepackage[all]{xy}

\newtheorem{theorem}{Theorem}[section]

\newtheorem{proposition}[theorem]{Proposition}
\newtheorem{corollary}[theorem]{Corollary}

\theoremstyle{definition}

\theoremstyle{remark}
\newtheorem{remark}[theorem]{Remark}

\numberwithin{equation}{section}

\begin{document}

\title{Purely inseparable points on curves}

\author{Damian R\"ossler}
\address{Mathematical Institute\\ 
University of Oxford\\
Andrew Wiles Building\\
Radcliffe Observatory Quarter\\
Woodstock Road\\
Oxford\\
OX2 6GG\\
United Kingdom}
\curraddr{Mathematical Institute\\ 
University of Oxford\\
Andrew Wiles Building\\
Radcliffe Observatory Quarter\\
Woodstock Road\\
Oxford\\
OX2 6GG\\
United Kingdom}
\email{rossler@maths.ox.ac.uk}




\dedicatory{To Gerhard Frey on the occasion of his $75^{\tiny th}$ birthday.}


\begin{abstract}
We give effective upper bounds for the number of purely inseparable points on non isotrivial curves over function fields 
of positive characteristic and of transcendence degree one. These bounds depend on the genus of the curve, the genus of the function field and the number of 
points of bad reduction of the curve. 
\end{abstract}

\maketitle

\section{Introduction}

I am very happy to have the possibility to dedicate this article to Gerhard Frey, who made so many interesting contributions to number theory. An electronic correspondence with him on 
the subject of Kim's article \cite{Kim-Purely} prompted me to look for effective upper bounds for the number of purely inseparable points on curves and the present text summarises what I found.

Let $k$ be an algebraically closed field of characteristic $p>0$. 
Let $K$ be the function field of a smooth, proper and geometrically connected 
curve $B$ over $k$. Let $C$ be a smooth, proper and geometrically connected  
curve over $K$. Let $g_B$ (resp. $g_C$) be the genus of $B$ (resp. $C$). 
We shall denote by $A$ the Jacobian of $C$ over $K$. 
We let $\Sigma\subseteq B$ be the 
set of closed points of bad reduction of $A$ and we let $N=\#\Sigma$. 
Suppose that $g_C>0$ and that $C(K)\not=\emptyset$. We suppose that $C$ is not isotrivial over $K$. By definition, this means that there is no finite  
field extension $L|K$, such that $C_L\simeq C_{0,L}$, where $C_0$ is a smooth curve over $k$. 
We choose an algebraic closure $\bar K$ of $K$ and we denote by $K_{\rm ins}:=\bigcup_{i\geq 0}K^{p^{-i}}\subseteq\bar K$ the maximal purely inseparable extension of $K$ in $\bar K$. This is also called
the perfection of $K$. The elements of $C(K_{\rm ins})$ will be called the {\it perfect points} or {\it purely inseparable points} of $C$. 

We shall prove

\begin{theorem}
{\rm (a)} If $g_C>1$ then $C(K_{\rm ins})$ is a finite set.

\smallskip
{\rm (b)} Suppose that the Jacobian $A$ of $C$ has semiabelian reduction over $B.$ 

Let $\ell_0\geq 1$ be such that 
\begin{equation}
p^{\ell_0}-p^{\ell_0-1}g_C(g_C-1){\max}\{2g_B-2,0\}-g_C(2g_B-2)-g_C N>0.
\label{BH00}
\end{equation}
Let $\ell$ be the smallest integer such that $\ell\geq \ell_0$ and $p^{\ell}>2g_C-2$. Then:

- the set $C(K_{\rm ins})\smallsetminus C(K^{p^{-\ell+1}})$ 
is empty; 

- we have $\#(C(K_{\rm ins})\smallsetminus C(K^{p^{-\ell_0+1}}))\leq 2g_C-2.$
\label{mainth}
\end{theorem}
Note that inequality \eqref{BH00} is a condition on both $p$ and $\ell_0$. 
For a fixed $\ell_0$ (for example $\ell_0=1$) it is satisfied if $p$ is large enough. 
On the other hand, if $$g_C(g_C-1){\max}\{2g_B-2,0\}<p$$ then it will be satisfied for any sufficiently 
large $\ell_0$. In particular, we have the
\begin{corollary}
Suppose that the Jacobian $A$ of $C$ has semiabelian reduction over $B$. 
If $g_B>0$ and 
$$
p>{\max}\{2g_C-2,g^2_C(2g_B-2)+g_C N\}
$$
then the set $C(K_{\rm ins})\smallsetminus C(K)$ is empty. If $g_B=0$ and 
$$
p>{\max}\{2g_C-2,g_C(N-2)\}
$$
then the set $C(K_{\rm ins})\smallsetminus C(K)$ is empty.
\label{maincor}
\end{corollary}
\begin{remark}\rm\label{MC} Theorem \ref{mainth} (a) also follows from the Claim on page 665 of  \cite{Kim-Purely}. \end{remark}
\begin{remark}\label{ML}\rm  Let $L|K$ be a finite field extension. If $g_C>1$ then 
$C(L)$ is finite. This is a consequence of (a strong form of) the Mordell conjecture over function fields 
(see \cite{Samuel-Comp}).\end{remark}
\begin{remark}\rm The Jacobian $A$ of $C$ has semiabelian reduction over $B$ if and only if the minimal 
regular model of $C$ over $B$ has semistable reduction. See eg \cite[4.]{Abbes-RSS} for a lucid 
exposition.\end{remark}

{\bf Outline of the proof of Theorem \ref{mainth}.} The basic idea behind the proof of
Theorem \ref{mainth} comes from M. Kim's article \cite{Kim-Purely}. However, unlike him we make use of 
the connected component of the N\'eron model of $A$ and of the theory of semistable sheaves in positive characteristic. By contrast Kim 
uses the theory of heights and the fact that the relative dualising sheaf of the minimal regular model of $C$ over $B$ is 
big over $k$. This follows from \cite[p. 53, after Th. 2]{Szpiro-PNDR} (private
communication from M. Kim to the author). The proof is in several steps. 

Step (1) We associate with a purely inseparable point of degree $p^\ell$ over $K$ a non vanishing map of $K$-vector spaces 
$F_K^{\circ\ell,*}(\omega_{K})\to\Omega_{K/k}$. Here $\omega$ is the Hodge bundle of the connected component of the N\'eron model of $A$, ie its sheaf of differentials restricted to the zero section. 
This map is basically just a restriction map (see details below).

Step (2) We quote a result from \cite{Rossler-PII} stating that the poles of this map only have 
logarithmic singularities along the points of bad reduction of $A$, so that it extends to a morphism of vector bundles
$F_B^{\circ\ell,*}(\omega)\to\Omega_{B/k}(\Sigma)$. 

Step (3) From a result of Szpiro (see below for references), we know that 
${\rm deg}_B(\omega)>0$ and thus $\omega$ contains a semistable subsheaf $\omega_1$ of positive slope. 
If the semistability of $\omega_1$ were preserved under iterated Frobenius pull-backs, 
we could conclude that if $\ell$ is large enough, the image of $F_B^{\circ\ell,*}(\omega_1)$ 
in $\Omega_{B/k}$ would vanish. Choosing a non-zero differential form $\lambda\in\omega_{1,K}$, 
we could conclude that $\lambda$ vanishes on all the purely inseparable points of sufficiently large degree. 
This would immediately show that there is only finitely many such points. However, we cannot assume that the semistability of $\omega_1$ is preserved under iterated Frobenius pull-backs.

Step (4) To deal with this last issue, we use a result of Langer, which gives a numerical 
measure for the lack of semistability of Frobenius pull-backs of semistable sheaves. 
Using this, up to making assumption on $p$ and $\ell$, we can basically pull through 
Step (3).

The structure of the paper is as follows. In section \ref{SS} we recall the results from the theory of semistability of vector bundles on curves that we shall need. In section \ref{SP}, we prove 
Theorem \ref{mainth}.

{\bf Notation.} If $S$ is a scheme of characteristic $p$, we write $F_S$ for the 
absolute Frobenius endomorphism of $S$. If $T$ is a scheme over $S$ and $n\geq 0$, 
we write $T^{(p^n)}$ for the base-change of $T$ by $F_S^{\circ n}$. For any integral scheme $W$ we write $\kappa(W)$ for the local ring at the generic point of $W$ (which is a field). If $f:X\to Y$ is a morphism of schemes, we write $\Omega_f$ or $\Omega_{X/Y}$ for the sheaf of differentials of the morphism $f$. 

{\bf Acknowledgments.} Many thanks to the anonymous referee for his useful remarks and suggestions of improvement. I am also grateful to Moshe Jarden for his comments.

\section{Semistable sheaves}

\label{SS}

We shall need some terminology and results from the theory of semistability of vector bundles over curves.

If $V$ be a non zero vector bundle (ie a coherent locally free sheaf) on $B$, we write as is customary
$$
\mu(V):={\rm deg}(V)/{\rm rk}(V)
$$
where ${\rm deg}(V)$ is the degree of the line bundle $\det(V)$ and 
${\rm rk}(V)$ is the rank of $V$. Recall that the degree ${\rm deg}(L)$ of a line bundle $L$ on $B$ can be computed as 
follows. Choose any rational section $\sigma$ of $L$; the degree of $L$ is the number of zeroes of $\sigma$ minus the number of poles of $\sigma$ (both counted with multiplicities); this quantity does not depend on the choice of $\sigma$. See \cite[II, par. 6, Cor. 6.10]{Hartshorne-Algebraic} for details. 

The quantity 
$\mu(V)$ is called the {\it slope} of $V$. 
Recall that a non zero vector bundle $V$ on $B$ is called 
semistable if for any non zero coherent subsheaf $W\subseteq V$, we have 
$\mu(W)\leqslant\mu(V)$ (recall that a coherent subsheaf of $V$ is automatically 
locally free, because it is torsion free and $B$ is a Dedekind scheme).
 
Let $V$ be a non zero vector bundle on $B$. There is a unique filtration 
by coherent subsheaves
$$
0=V_{0}\subsetneq V_1\subsetneq V_2\subsetneq\dots\subsetneq V_{{\rm hn}(V)}=V
$$
such that all the sheaves $V_i/V_{i-1}$ ($1\leqslant i\leqslant{\rm hn}(V)$) are (vector bundles and) semistable and 
such that the sequence $\mu(V_i/V_{i-1})$ is strictly decreasing. 
This 
filtration is called the {\it Harder-Narasimhan filtration} of $V$ (short-hand: {\rm HN}\,filtration). One then defines 
$$
V_{\min}:=V/V_{{\rm hn}(V)-1},\,
V_{\max}(V):=V_1
$$
and
$$\mu_{\max}(V):=\mu(V_1),\,\mu_{\min}(V):=\mu(V_{\min}).
$$
One basic property of semistable sheaves that we shall use 
is the following. If 
$V$ and $W$ are non zero vector bundles on $B$ and $\mu_{\min}(V)>\mu_{\max}(W)$ then 
${\rm Hom}_B(V,W)=0$. This follows from the definitions. 

See  \cite[chap. 5]{Brenner-Herzog-Villamayor-Three} (for instance) for all these notions. 

If $V$ is a non zero vector bundle on $B$ 
we say that $V$ is {\it Frobenius semistable} if $F^{\circ r,*}_B(V)$ is semistable for all $r\geq 0$.  The terminology {\it strongly semistable} also appears in the literature. 

\begin{theorem} Let $V$ be a non zero vector bundle on $B$. 
There is an $\ell_0=\ell_0(V)\in{\mathbb N}$ such that the quotients of the Harder-Narasimhan filtration of $F^{\circ\ell_0,*}_B(V)$ are 
all Frobenius semistable. 
\label{SSth}
\end{theorem}
\begin{proof}
See eg \cite[Th. 2.7, p. 259]{Langer-Semistable}.
\end{proof}

Theorem \ref{SSth} shows in particular that the following definitions : 
$$
\bar\mu_{\min}(V):=\lim_{\ell\to\infty}\mu_{\min}(F^{\circ\ell,*}_B(V))/p^\ell,
$$
$$
\bar\mu_{\max}(V):=\lim_{\ell\to\infty}\mu_{\max}(F^{\circ\ell,*}_B(V))/p^\ell,
$$
make sense if $V$ is a non zero vector bundle on $B$. 

Note that since $F_B$ is faithfully flat, this implies that if 
$V$ and $W$ are non zero vector bundles on $B$ and $\bar\mu_{\min}(V)>\bar\mu_{\max}(W)$ then we also have 
${\rm Hom}_B(V,W)=0$. 

For use below, we note that one can show that the 
sequence $\mu_{\min}(F^{\circ\ell,*}_B(V))/p^\ell$ 
(resp. $\mu_{\max}(F^{\circ\ell,*}_B(V))/p^\ell$) is weakly decreasing (resp. increasing). 
See \cite[p. 258, before Cor. 2.5]{Langer-Semistable}. 

We shall need the following numerical estimate. 
To formulate it, let $V$ be a non zero vector bundle on $B$ and let 
$$
\alpha(V):={\max}\{\mu_{\min}(V)-\bar\mu_{\min}(V),\bar\mu_{\max}(V)-\mu_{\max}(V)\}.
$$
By the preceding remark, we have $\alpha(V)\geq 0.$ 
\begin{theorem}[Langer]
We have
$$
\alpha(V)\leqslant{{\rm rk}(V)-1\over p}{\max}\{\bar\mu_{\max}(\Omega_{B/k}),0\}.
$$
\label{laneff}
\end{theorem}
\begin{proof} See \cite[Cor. 6.2]{Langer-Semistable}.\end{proof}

\section{Proof of Theorem \ref{mainth}}

\label{SP}

We let ${\mathcal A}$ be the connected component of the identity of the N\'eron model 
of the Jacobian $A$ of $C$ over $B$. We suppose throughout the proof that 
${\mathcal A}$ is a semiabelian scheme. In (b), this is an assumption and in (a), we may assume 
it without restriction of generality, since $A$ will acquire semiabelian reduction on a finite 
extension of $K$ by a classical result of Grothendieck. 

We shall write $\omega=\omega_{\mathcal A}$ for the 
restriction of $\Omega_{{\mathcal A}/B}$ to $B$ by the zero section of ${\mathcal A}$. Under the assumptions of Theorem \ref{mainth}, we have ${\rm deg}_B(\omega)>0$. 
This is a consequence of  \cite[Th. 1]{Szpiro-PNDR} and \cite[9.5, Th. 4 (b) and 8.4, Th. 1]{BLR-Neron}, together with Grothendieck duality. 


Let us write $f:A\to{\rm Spec}(K)$ (resp. $\pi:{\mathcal A}\to B$) for the structural morphism. In particular 
$f=\pi_K$. 

Choose once and for all an element of $C(K)$ to determine an embedding of $C$ into its Jacobian 
and write $\iota:C\hookrightarrow A$ for this embedding. 
Write $g:C\to {\rm Spec}(K)$ for the structural morphism.

Fix $\ell_0\geq 1$. 

Let $P\in C(K_{\rm ins})\smallsetminus C(K^{p^{-\ell_0+1}})$. Let 
$\ell\geq 0$ be such that $[\kappa(P):K]=p^\ell$ (in other words 
we have $P\in C(K^{p^{-\ell}})\smallsetminus C(K^{p^{-\ell+1}})$). Note that by construction we 
have $\ell\geq \ell_0$. The point $P$ corresponds by definition to an element 
$\widetilde{P}\in C^{(p^\ell)}(K)$ and thus we obtain a commutative diagram

$$
\xymatrix{
C^{(p^{\ell})}\ar[dr]^{g^{(p^\ell)}}\ar[r] & C\ar[r]^{\iota}\ar[rd]^{g}& A\ar[d]^{f}\\
 & {\rm Spec}(K)\ar[r]^{F^{\circ\ell}_K\,\,\,}\ar@/^1pc/[ul]^{\widetilde{P}}\ar[u]^{\bf P} & {\rm Spec}(K)
}
$$

where ${\bf P}$ is a closed immersion and the parallelogram on the left is cartesian. We used the short-hand $F^{\circ\ell}_K$ for 
$F^{\circ\ell}_{{\rm Spec}(K)}.$ This diagram induces morphisms of differentials
\begin{equation}
{\bf P}^*(\iota^*(\Omega_f))\to {\bf P}^*(\Omega_g)\to\Omega_{F_K^{\circ\ell}}
\label{IE}
\end{equation}

where the morphism ${\bf P}^*(\iota^*(\Omega_f))\to {\bf P}^*(\Omega_g)$ is induced by the 
natural morphism $$\iota^*(\Omega_f)\to\Omega_g.$$ In particular, both arrows in 
\eqref{IE} are surjective maps of $K$-vector spaces, since $\iota$ and ${\bf P}$ are closed immersions.

Recall that we have a canonical isomorphism $\Omega_\pi\simeq\pi^*(\omega_{\mathcal A})$, since ${\mathcal A}$ is a group scheme over $B$. 
Thus we have a canonical isomorphism $$F_K^{\circ\ell,*}(\omega_K)\simeq {\bf P}^*(\iota^*(\Omega_f)).$$

Composing this isomorphism with the two arrows in \eqref{IE}, we obtain 
a map 
$$F_K^{\circ\ell,*}(\omega_K)\to\Omega_{F_K^{\circ\ell}},$$
which obviously depends on $P$.

We shall now make use of the following result. 

\begin{proposition}
The map $$F_K^{\circ\ell,*}(\omega_K)\to\Omega_{F_K^{\circ\ell}}$$ extends 
to a morphism of vector bundles
$$
\phi_P:F_B^{\circ\ell,*}(\omega)\to\Omega_{F_B^{\circ\ell}}(\Sigma).
$$
\label{cruxlem}
\end{proposition}
\begin{proof} See \cite[Lemma B.2]{Rossler-PII}.\end{proof}
\begin{remark}\rm Note that the proof of Proposition \ref{cruxlem} is not elementary. It makes use 
of logarithmic differentials and of the semistable compactifications of semiabelian schemes constructed 
by Faltings, Chai and Mumford.\end{remark}
Write $b:B\to{\rm Spec}(k)$ for the structural morphism. 
Note that we have an exact sequence 
$$
F_B^{\circ\ell,*}(\Omega_b)\stackrel{F_B^{\circ\ell,*}}{\to}\Omega_{F_k^{\circ\ell}\circ b}\to \Omega_{F_B^{\circ\ell}}\to 0
$$
and a standard computation shows that the first arrow in the sequence vanishes. Hence we have a canonical isomorphism
$$
\Omega_{F_k^{\circ\ell}\circ b}\simeq\Omega_{F_B^{\circ\ell}}
$$
and since $F_k^{\circ\ell}$ is an isomorphism (because $k$ is perfect), we have a canonical isomorphism 
$\Omega_{F_k^{\circ\ell}\circ b}\simeq \Omega_{b}$ so that we finally get a canonical isomorphism
$$
\Omega_{B/k}=\Omega_{b}\simeq\Omega_{F_B^{\circ\ell}}.
$$
In particular, $\Omega_{F_B^{\circ\ell}}$ is a line bundle. 
This implies in particular that the arrow 
$${\bf P}^*(\Omega_g)\to\Omega_{F_K^{\circ\ell}}$$ is an isomorphism
because both ${\bf P}^*(\Omega_g)$ and $\Omega_{F_K^{\circ\ell}}$ are $K$-vector spaces of dimension one 
and the arrow is surjective. 

We shall now try to exhibit a differential form on $C$, which vanishes on the non rational 
purely inseparable points of $C$.

Recall that we know that ${\rm deg}_B(\omega)>0$. In particular, we must have $\mu_{\max}(\omega)>0$. Let $\omega_1$ be the first step of the HN filtration of $\omega.$ By definition, we have $\mu(\omega_1)=\mu_{\max}(\omega)$. Suppose now
until further notice that  the inequality
\begin{equation}
p^{\ell_0}(\mu(\omega_1)-\alpha(\omega_1))>\bar\mu_{\max}(\Omega_{F_B^{\circ\ell}}(\Sigma))
\label{BH0}
\end{equation}
holds. Since $\Omega_{F_B^{\circ\ell}}\simeq\Omega_{B/k}$, we have 
$$\bar\mu_{\max}(\Omega_{F_B^{\circ\ell}}(\Sigma))=2g_B-2+N.$$ 
Thus inequality \eqref{BH0} is equivalent to the inequality
\begin{equation}
p^{\ell_0}(\mu(\omega_1)-\alpha(\omega_1))>2g_B-2+N.
\label{BH}
\end{equation}
We then have
$$
p^{\ell}\bar\mu_{\min}(\omega_1)\geq p^{\ell_0}\bar\mu_{\min}(\omega_1) >\bar\mu_{\max}(\Omega_{F_B^{\circ\ell}}(\Sigma))
$$
 (by the definition of $\alpha(\bullet)$) and thus
$$
\bar\mu_{\min}(F_B^{\circ\ell,*}(\omega_1))=
p^{\ell}\bar\mu_{\min}(\omega_1)>\bar\mu_{\max}(\Omega_{F_B^{\circ\ell}}(\Sigma)).
$$
This implies that $\phi_P(F_B^{\circ\ell,*}(\omega_1))=0$ (see before Theorem \ref{laneff}).

On the other hand, if 
$\phi_P(F_B^{\circ\ell,*}(\omega_1))=0$ then from the fact that the arrow ${\bf P}^*(\Omega_g)\to\Omega_{F_K^{\circ\ell}}$ is an isomorphism, we can conclude that  
the arrow
\begin{equation}
{\bf P}^*(\iota^*(f^*(\omega_{1,K}))\to {\bf P}^*(\Omega_g)
\label{IA}
\end{equation}
vanishes. 

Now fix $\lambda\in\omega_{1,K}$ with $\lambda\not=0$. 
Recall that the canonical map
$$
H^0(A,\Omega_f)\to H^0(C,\Omega_g)
$$
induced by $\iota$ is an isomorphism, because $A$ is the Jacobian of $C$. 
Via the identification $\omega_K\simeq H^0(A,\Omega_f)$, the element $\lambda$ corresponds 
an element of $H^0(A,\Omega_f)$ and thus to a non zero element $\lambda_C$ of 
$H^0(C,\Omega_g)$. Since the arrow \eqref{IA} vanishes, we conclude that $P$ is contained in  
the vanishing locus $Z(\lambda_C)$ of $\lambda_C$. 
We have
$$
{\rm deg}(\Omega_g)\geq\sum_{z\in Z(\lambda_C)}[\kappa(z):K]\geq\#Z(\lambda_C)
$$
In particular, since $P\in C(K_{\rm ins})\smallsetminus C(K^{p^{-\ell_0+1}})$ was arbitrary, we see that 
\begin{equation*}
\#(C(K_{\rm ins})\smallsetminus C(K^{p^{-\ell_0+1}}))\leq\#Z(\lambda_C)\leq{\rm deg}(\Omega_g)=2g_C-2.
\end{equation*}
Furthermore, we see that
\begin{equation*}
p^\ell=[\kappa(P):K]\leq 2g_C-2.
\end{equation*}

{\it To summarise}:

Let $\ell_0\geq 1$. 

(A) If $p^{\ell_0}(\mu(\omega_1)-\alpha(\omega_1))>2g_B-2+N$ then $\#(C(K_{\rm ins})\smallsetminus C(K^{p^{-\ell_0+1}}))\leq 2g_C-2.$

(B) If $\ell\geq\ell_0$,  $P\in C(K^{p^{-\ell}})\smallsetminus C(K^{p^{-\ell+1}})$ and 
$p^{\ell_0}(\mu(\omega_1)-\alpha(\omega_1))>2g_B-2+N$ then 
$p^\ell\leq 2g_C-2.$

We now prove (a) in Theorem \ref{mainth}. Note that to prove (a), we may by definition 
replace $C$ by $C^{(p^c)}$ for any $c\geq 0$. By Theorem \ref{SSth}, we may thus assume that the quotients of the {\rm HN}\,filtration of 
$\omega_{\mathcal A}$ are Frobenius semistable. In particular, we may assume that $\omega_1$ is 
Frobenius semistable and thus that 
$\alpha(\omega_1)=0$. Now choose $\ell_0\geq 1$ so that 
$$p^{\ell_0}>(2g_B-2+N)/\mu(\omega_1).$$
Then  (A) above implies that 
$C(K_{\rm ins})\smallsetminus C(K^{p^{-\ell_0+1}})$ is finite. By the Mordell conjecture over function fields 
(see remark \ref{ML}) $C(K^{p^{-\ell_0+1}})$ is also finite and this concludes the proof of (a).

We now turn to the proof of (b). We are going to seek numerical conditions ensuring that 
inequality \eqref{BH} holds, ie that 
$$p^{\ell_0}(\mu(\omega_1)-\alpha(\omega_1))>2g_B-2+N.
$$
Note that by Theorem \ref{laneff} we have 
$$
\alpha(\omega_1)\leqslant{{\rm rk}(\omega_1)-1\over p}{\max}\{\bar\mu_{\max}(\Omega_{B/k}),0\}.
$$
Furthermore, we clearly have ${\rm rk}(\omega_1)\leq {\rm rk}(\omega)=g_C$ and 
$$\bar\mu_{\max}(\Omega_{B/k})={\rm deg}(\Omega_{B/k})=2g_B-2.$$ Thus we have 
$$
\alpha(\omega_1)\leqslant{(g_C-1){\max}\{2g_B-2,0\}\over p}
$$
Also we have the simple lower bound $\mu(\omega_1)\geq 1/g_C$. 

For \eqref{BH} to hold it is thus sufficient to have
$$
p^{\ell_0}\Big(1/g_C-{(g_C-1){\max}\{2g_B-2,0\})\over p}\Big)>2g_B-2+N
$$
or in other words, to have
\begin{equation}
p^{\ell_0}-p^{\ell_0-1}g_C(g_C-1){\max}\{2g_B-2,0\}-g_C(2g_B-2)-g_C N>0.
\label{FEQ}
\end{equation}
From (A) above we see that if $\ell_0\geq 1$ and inequality \eqref{FEQ} holds then $$\#(C(K_{\rm ins})\smallsetminus C(K^{p^{-\ell_0+1}}))\leq 2g_C-2.$$ This is 
the second statement in (b). The first statement in (b) is now a direct consequence of (B).


\begin{bibdiv}
\begin{biblist}

\bib{Abbes-RSS}{article}{
   author={Abbes, Ahmed},
   title={R\'{e}duction semi-stable des courbes d'apr\`es Artin, Deligne,
   Grothendieck, Mumford, Saito, Winters, $\ldots$},
   language={French},
   conference={
      title={Courbes semi-stables et groupe fondamental en g\'{e}om\'{e}trie
      alg\'{e}brique},
      address={Luminy},
      date={1998},
   },
   book={
      series={Progr. Math.},
      volume={187},
      publisher={Birkh\"{a}user, Basel},
   },
   date={2000},
   pages={59--110},
   review={\MR{1768094}},
}

\bib{BLR-Neron}{book}{
   author={Bosch, Siegfried},
   author={L\"{u}tkebohmert, Werner},
   author={Raynaud, Michel},
   title={N\'{e}ron models},
   series={Ergebnisse der Mathematik und ihrer Grenzgebiete (3) [Results in
   Mathematics and Related Areas (3)]},
   volume={21},
   publisher={Springer-Verlag, Berlin},
   date={1990},
   pages={x+325},
   isbn={3-540-50587-3},
   review={\MR{1045822}},
   doi={10.1007/978-3-642-51438-8},
}

\bib{Brenner-Herzog-Villamayor-Three}{collection}{
  author={Brenner, Holger},
  author={Herzog, J{\"u}rgen},
  author={Villamayor, Orlando},
  title={Three lectures on commutative algebra},
  series={University Lecture Series},
  volume={42},
  note={Lectures from the Winter School on Commutative Algebra and Applications held in Barcelona, January 30--February 3, 2006; Edited by Gemma Colom\'e-Nin, Teresa Cortadellas Ben\'\i tez, Juan Elias and Santiago Zarzuela},
  publisher={American Mathematical Society},
  place={Providence, RI},
  date={2008},
  pages={vi+190},
  isbn={978-0-8218-4434-2},
  isbn={0-8218-4434-2},
}

\bib{Hartshorne-Algebraic}{book}{
   author={Hartshorne, Robin},
   title={Algebraic geometry},
   note={Graduate Texts in Mathematics, No. 52},
   publisher={Springer-Verlag, New York-Heidelberg},
   date={1977},
   pages={xvi+496},
   isbn={0-387-90244-9},
   review={\MR{0463157}},
}


\bib{Huybrechts-Lehn-The-geometry}{book}{
  author={Huybrechts, Daniel},
  author={Lehn, Manfred},
  title={The geometry of moduli spaces of sheaves},
  series={Aspects of Mathematics, E31},
  publisher={Friedr. Vieweg \& Sohn, Braunschweig},
  date={1997},
  pages={xiv+269},
  isbn={3-528-06907-4},
  doi={10.1007/978-3-663-11624-0},
}

\bib{Kim-Purely}{article}{
   author={Kim, Minhyong},
   title={Purely inseparable points on curves of higher genus},
   journal={Math. Res. Lett.},
   volume={4},
   date={1997},
   number={5},
   pages={663--666},
   issn={1073-2780},
   review={\MR{1484697}},
   doi={10.4310/MRL.1997.v4.n5.a4},
}

\bib{Langer-Semistable}{article}{
  author={Langer, Adrian},
  title={Semistable sheaves in positive characteristic},
  journal={Ann. of Math. (2)},
  volume={159},
  date={2004},
  number={1},
  pages={251--276},
  issn={0003-486X},
  doi={10.4007/annals.2004.159.251},
}

\bib{Rossler-PII}{article}{
author={R\"ossler, Damian},
title={On the group of purely inseparable points of an abelian variety defined over a function field of positive characteristic II},
status={arXiv:1702.07142}}


\bib{Samuel-Comp}{article}{
   author={Samuel, Pierre},
   title={Compl\'{e}ments \`a un article de Hans Grauert sur la conjecture de
   Mordell},
   language={French},
   journal={Inst. Hautes \'{E}tudes Sci. Publ. Math.},
   number={29},
   date={1966},
   pages={55--62},
   issn={0073-8301},
   review={\MR{0204430}},
}

\bib{Szpiro-PNDR}{article}{
   author={Szpiro, Lucien},
   title={Propri\'{e}t\'{e}s num\'{e}riques du faisceau dualisant relatif},
   language={French},
   note={Seminar on Pencils of Curves of Genus at Least Two},
   journal={Ast\'{e}risque},
   number={86},
   date={1981},
   pages={44--78},
   issn={0303-1179},
   review={\MR{3618571}},
}








\end{biblist}
\end{bibdiv}
\end{document}